\documentclass[sts]{imsart}
\input{commands.tex}

\begin{document}
\begin{frontmatter}
\title{Statistical bounds for entropic optimal transport: sample complexity and the central limit theorem}
\runtitle{Statistical bounds for entropic optimal transport}

\begin{aug}
\author{\fnms{Gonzalo}~\snm{Mena}\thanksref{t1}\ead[label=mena]{gomena@fas.harvard.edu}}
\and
\author{\fnms{Jonathan}~\snm{Weed}\thanksref{t2}\ead[label=jon]{jweed@mit.edu}}

\affiliation{Harvard University\\Massachusetts Institute of Technology}
\thankstext{t1}{This work was supported by a Harvard Data Science Initiative Fellowship.}
\thankstext{t2}{This work was supported in part by the Josephine de K\'arm\'an Fellowship.}

\address{{Gonzalo Mena}\\
{Department of Statistics} \\
{Harvard University}\\
{1 Oxford St \#7,}\\
{Cambridge, MA 02138-4307, USA}\\
\printead{mena}
}

\address{{Jonathan Weed}\\
{Department of Mathematics} \\
{Massachusetts Institute of Technology}\\
{77 Massachusetts Avenue,}\\
{Cambridge, MA 02139-4307, USA}\\
\printead{jon}
}

\runauthor{Mena and Weed}
\end{aug}

\begin{abstract}
We prove several fundamental statistical bounds for entropic OT with the squared Euclidean cost between subgaussian probability measures in arbitrary dimension.
First, through a new sample complexity result we establish the rate of convergence of entropic OT for empirical measures.
Our analysis improves exponentially on the bound of Genevay et al.~(2019) and extends their work to unbounded measures.
Second, we establish a central limit theorem for entropic OT, based on techniques developed by Del Barrio and Loubes~(2019).
Previously, such a result was only known for finite metric spaces.
As an application of our results, we develop and analyze a new technique for estimating the entropy of a random variable corrupted by gaussian noise.
\end{abstract}

\begin{keyword}[class=AMS]
\kwd[]{Statistics}
\end{keyword}
\begin{keyword}[class=KWD]
Optimal Transport, Entropic Regularization, Central Limit Theorem, Sample Complexity
\end{keyword}

\end{frontmatter}

\section{Introduction}
Optimal transport is an increasingly popular tool for the analysis of large data sets in high dimension, with applications in domain adaptation~\citep{CouFlaTui14, CouFlaTui17}, image recognition~\citep{LiWanZha13, RubTomGui00, SanLin11}, and word embedding~\citep{AlvJaa18,GraJouBer18}.
Its flexibility and simplicity have made it an attractive choice for practitioners and theorists alike, and its ubiquity as a machine learning tool continues to grow~\citep[see, e.g.,][for surveys]{Peyre2019,KolParTho17}.

Much of the recent interest in optimal transport has been driven by algorithmic advances, chief among them the popularization of entropic regularization as a tool of solving large-scale OT problems quickly~\citep{Cuturi2013}.
Not only has this proposal been shown to yield near-linear-time algorithms for the original optimal transport problem~\citep{Altschuler2017}, but it also appears to possess useful \emph{statistical} properties which make it an attractive choice for machine learning applications~\citep{Rigollet2018,Genevay2017,SchShuTab17,Montavon2016}.
For instance, in a recent breakthrough work, \citet{Genevay2018} established that even though the empirical version of standard OT suffers from the ``curse of dimensionality''~\citep[see, e.g.][]{Dud69}, the empirical version of entropic OT always converges at the parametric $1/\sqrt n$ rate for compactly supported probability measures.
This result suggests that entropic OT may be significantly more useful than unregularized OT for inference tasks when the dimension is large.
However, obtaining rigorous guarantees for the performance of entropic OT in practice requires a more thorough understanding of its statistical behavior.

\subsection{Summary of contributions}
We prove new results on the relation between the population and empirical version of the entropic cost, that is, between $S(P,Q)$ and $S(P_n,Q_n)$ (defined in Section~\ref{sec:background}, below).
These results give the first characterization of the large-sample behavior of entropic OT for unbounded probability measures in arbitrary dimension.
Specifically, we obtain: \textbf{(i)} New sample complexity bounds on $E |S(P, Q) - S(P_n, Q_n)|$ which improve on the results of \citet{Genevay2018} by an exponential factor and which apply to unbounded measures (Section~\ref{sec:sample}).
\textbf{(ii)} A central limit theorem characterizing the fluctuations $S(P_n, Q_n) - E S(P_n, Q_n)$ when $P$ and $Q$ are subgaussian (Section~\ref{sec:CLT}).
Such a central limit theorem was previously only known for probability measures supported on a finite number of points~\citep{BigCazPap17,KlaTamMun18}.
We use completely different techniques, inspired by recent work of~\cite{Del2019}, to prove our theorem for general subgaussian distributions.

As an application of our results, we show how entropic OT can be used to shed new light on the entropy estimation problem for random variables corrupted by subgaussian noise (Section~\ref{sec:entropy}).
This problem has gained recent interest in machine learning \citep{Goldfeld2018a,Goldfeld2018b} as a tool for obtaining a theoretically sound understanding of the \textit{Information Bottleneck Principle} in deep learning \citep{Tishby2015}.
We design and analyze a new estimator for this problem based on entropic OT.

Finally, we provide simulations which give empirical validation for our theoretical claims (Section~\ref{sec:experiments}).

\subsection{Background and preliminaries}\label{sec:background}
Let $P,Q\in \mathcal{P}(\mathbb{R}^d)$ be two probability measures and let $P_n$ and $Q_n$ be the empirical measures from the independent samples $\{X_i\}_{i\leq n}\sim P^n$ and \mbox{$\{Y_i\}_{i\leq n}\sim Q^n$}. We define the squared Wasserstein distance between $P$ and $Q$ \citep{Villani2008} as follows:
\begin{equation}\label{eq:ot}
W_2^2(P, Q) := \inf_{\pi\in \Pi(P,Q)} \left[\int_{\mx\times\my} \frac 1 2 \|x - y\|^2 \,\dd\pi(x,y)\right]\,,\end{equation}
where $\Pi(P,Q)$ is the set of all joint distributions with marginals equal to $P$ and $Q$, respectively.
We focus on a entropy regularized version of the above cost \citep{Cuturi2013,Peyre2019}, defined as

 \begin{equation}\label{eq:sink} S_\epsilon(P,Q) := \inf_{\pi\in \Pi(P,Q)} \left[\int_{\mx\times\my} \frac 12 \|x - y\|^2 \,\dd\pi(x,y) + \epsilon H(\pi\lvert P\otimes Q)\right]\,,
 \end{equation}
where $H(\alpha\lvert \beta)$ denotes the relative entropy between probability measures $\alpha$ and $\beta$ defiend by $\int \log \frac{d\alpha}{d\beta}(x)d\alpha(x)$ if $\alpha \ll \beta$ and $+ \infty$ otherwise.
By rescaling the measures $P$ and $Q$ and the regularization parameter $\epsilon$, it suffices to analyze the case $\epsilon = 1$, which we denote by $S(P, Q)$.
Note that we consider the squared cost~$\frac 12 \| \cdot \|^2$ throughout.
While some of our results extend to other costs, we leave a full analysis of the general case to future work.

The general theory of entropic OT~\citep{Csiszar1975} implies that $S(P, Q)$ possesses a dual formulation:
\begin{multline}\label{eq:dual}
S(P, Q) = \sup_{f \in L_1(P), g \in L_1(Q)} \int f(x) \, \dd P(x) + \int g(y) \, \dd Q(y) \\ - \int e^{f(x) + g(y) - \frac{1}{2}||x- y||^2} \, \dd P(x) \dd Q(y) + 1\,,
\end{multline}
and that as long as $P$ and $Q$ have finite second moments, the supremum is attained at a pair of optimal potentials $(f, g)$ satisfying
\begin{equation}
\begin{split}
\int e^{f(x) + g(y) - \frac{1}{2}||x- y||^2} \, \dd Q(y)  & = 1 \quad \text{$P$-a.s.}\,, \\
\int e^{f(x) + g(y) - \frac{1}{2}||x- y||^2} \, \dd P(x)  & = 1 \quad \text{$Q$-a.s.}\,
\end{split}
\label{eq:dual-opt}
\end{equation}
Conversely, any $f \in L_1(P), g \in L_1(Q)$ satisfying~\eqref{eq:dual-opt} are optimal potentials.

We focus throughout on subgaussian probability measures.
We say that a distribution $P \in \mathcal{P}(\RR^d)$ is~$\sigma^2$-subgaussian for $\sigma \geq 0$ if $E_P e^{\frac{\|X\|^2}{2 d\sigma^2}} \leq  2$.
By Jensen's inequality, if $E_P e^{\frac{\|X\|^2}{2 d\sigma^2}} \leq C$ for any constant~$C \geq 2$, then $P$ is $C \sigma^2$-subgaussian. 
Note that if $P$ is subgaussian, then $E_P e^{v^\top X} < \infty$ for all $v \in \RR^d$.
Conversely, standard results~\citep[see, e.g.,][]{Vershynin2018} imply that our definition is satisfied if $E_P e^{u^\top X} \leq e^{\|u\|^2 \sigma^2/2}$ for all $u \in \RR^d$.

\section{Sample complexity for the entropic transportation cost for general subgaussian measures}
\label{sec:sample}
One rigorous statistical benefit of entropic OT is its \emph{sample complexity}, i.e., the minimum number of samples required for the empirical entropic OT cost $S(P_n, Q_n)$ to be an accurate estimate of $S(P, Q)$.
As noted above, unregularized OT suffers from the curse of dimensionality: in general, the Wasserstein distance $W_2^2(P_n, Q_n)$ converges to $W_2^2(P, Q)$ no faster than $n^{-1/d}$ for measures in $\RR^d$.
Strikingly, \cite{Genevay2018} established that the statistical performance of the entropic OT cost is significantly better.
They show:\footnote{We have specialized their result to the squared Euclidean cost.}
\begin{theorem}[{\citealp[Theorem 3]{Genevay2018}}]\label{thm:genevay}
Let $P$ and $Q$ be two probability measures on a bounded domain in $\RR^d$ of diameter $D$.
Then
\begin{equation}\label{geneeq}
\sup_{P,Q} E_{P,Q}|S_\epsilon(P,Q) - S_\epsilon(P_n,Q_n)|\leq K_{D,d} \left(1+\frac{1}{\epsilon^{\lfloor d/2 \rfloor}}\right) \frac{e^{D^2/\epsilon}}{\sqrt{n}}\,,
\end{equation}
where $K_{D, d}$ is a constant depending on $D$ and $d$.
\end{theorem}

This impressive result offers powerful evidence that entropic OT converges significantly faster than its unregularized counterpart.
The drawbacks of this result are that it applies only to bounded measures, and, perhaps more critically in applications, the rate scales \emph{exponentially} in $D$ and $1/\epsilon$, even in dimension $1$.
Therefore, while the qualitative message of Theorem~\ref{thm:genevay} is clear, it does not offer useful quantitative bounds as soon as the measure is unbounded or lies in a set of large diameter.

Our first theorem is a significant sharpening of Theorem~\ref{thm:genevay}.
We first state it for the case where $\epsilon = 1$.

\begin{theorem} \label{theo1}
If $P$ and $Q$ are $\sigma^2$-subgaussian, then
\begin{equation}
E_{P,Q}|S(P,Q) - S(P_n,Q_n)| \leq K_d (1 + \sigma^{\lceil 5d/2\rceil + 6})\frac 1{\sqrt n} \,.
\end{equation}
\end{theorem}

If we denote by $P^{\epsilon}$ and $Q^{\epsilon}$ the pushforwards of $P$ and $Q$ under the map $x \mapsto \epsilon^{-1/2} x$, then it is easy to see that
\begin{equation*}
S_\epsilon(P, Q) = \epsilon S(P^{\epsilon}, Q^{\epsilon})\,.
\end{equation*}
We immediately obtain the following corollary.
\begin{corollary}\label{cor:varying-ep}
If $P$ and $Q$ are $\sigma^2$-subgaussian, then
\begin{equation*}
E_{P,Q}|S_\epsilon(P,Q) - S_\epsilon(P_n,Q_n)| \leq K_d \cdot \epsilon \left(1 + \frac{\sigma^{\lceil 5d/2\rceil + 6}}{\epsilon^{\lceil 5d/4 \rceil + 3}} \right)\frac{1}{\sqrt n}\,.
\end{equation*}
\end{corollary}

If we compare Corollary~\ref{cor:varying-ep} with Theorem~\ref{thm:genevay}, we note that the polynomial prefactor in Corollary~\ref{cor:varying-ep} has higher degree than the one in Theorem~\ref{thm:genevay}, pointing to a potential weakness of our bound.
On the other hand, the exponential dependence on $D^2/\epsilon$ has completely disappeared.
Moreover, the brittle quantity $D$, finite only for compactly supported measures, has been replaced by the more flexible subgaussian variance proxy $\sigma^2$.

The improvements in Theorem \ref{theo1} are obtained via two different methods.
First, a simple argument allows us to remove the exponential term and bound the desired quantity by an empirical process, as in \cite{Genevay2018}.
Much more challenging is the extension to measures with unbounded support.
The proof technique of \cite{Genevay2018} relies on establishing uniform bounds on the derivatives of the optimal potentials, but this strategy cannot succeed if the support of $P$ and $Q$ is not compact.
We therefore employ a more careful argument based on controlling the H\"older norms of the optimal potentials on compact sets.
A chaining bound completes our proof.

In Proposition~\ref{prop:pointwise_derivative_bound} below (whose proof we defer to Appendix \ref{app:omitted}) we show that if $(f, g)$ is a pair of optimal potentials for $\sigma^2$-subgaussian distributions $P$ and $Q$, then we may control the size of $f$ and its derivatives.

\begin{proposition}\label{prop:pointwise_derivative_bound}
Let $P$ and $Q$ be $\sigma^2$-subgaussian distributions.
There exist optimal dual potentials $(f, g)$ for $P$ and $Q$ such that for any multi-index $\alpha$ with $|\alpha| = k$,
\begin{equation}\label{eq:pointwise_derivative_bound}
|D^\alpha (f - \frac 12\|\cdot\|^2)(x)| \leq C_{k, d}\left\{
\begin{array}{ll}
1 + \sigma^4  & \text{ $k = 0$} \\
 \sigma^{k} (\sigma + \sigma^2)^{k} & \text{ otherwise,} 
 \end{array}\right.
\end{equation}
if $\|x\| \leq \sqrt d \sigma$,
and
\begin{equation}\label{eq:pointwise_derivative_bound_2}
|D^\alpha (f - \frac 12\|\cdot\|^2)(x)| \leq C_{k, d}\left\{
\begin{array}{ll}
1 + (1+ \sigma^2) \|x\|^2 & \text{ $k = 0$}\\
 \sigma^{k} (\sqrt{\sigma \|x\|} + \sigma\|x\|)^{k} & \text{ otherwise,} 
 \end{array}\right.
\end{equation}
if $\|x\| > \sqrt d \sigma$, where $C_{k, d}$ is a constant depending only on $k$ and $d$.
\end{proposition}

We denote by $\mathcal F_\sigma$ the set of functions satisfying~\eqref{eq:pointwise_derivative_bound} and~\eqref{eq:pointwise_derivative_bound_2}.
The following proposition shows that it suffices to control an empirical process indexed by this set.
\begin{proposition} \label{noexp}
Let $P$, $Q$, and $P_n$ be $\tilde \sigma^2$-subgaussian distributions, for a possibly random $\tilde \sigma \in [0, \infty)$.
Then
\label{genevayimproved}
\begin{equation} |S(P_n,Q) - S(P,Q)|\leq 2 \sup_{u \in \mathcal F_{\tilde \sigma}} |E_P u - E_{P_n} u|\,.
\end{equation}
\end{proposition}
\begin{proof}
We define the operator $\mathcal{A}^{\alpha,\beta}(u,v)$ for the pair of probability measures $(\alpha,\beta)$ and functions $(u,v)\in L_1(\alpha)\otimes L_1(\beta)$ as:
$$\FA^{\alpha,\beta}(u,v) = \int u(x)\,\dd\alpha(x)+\int v(y)\,\dd\beta(y) -\int e^{u(x)+v(y)-\frac{1}{2}||x-y||^2}\,\dd\alpha(x)d\beta(y) + 1\,.$$
Denote by $(f_n,g_n)$ a pair of optimal potentials for $(P_n, Q)$ and $(f,g)$ for $(P,Q)$, respectively.
By Proposition~\ref{extend-potentials} in Appendix \ref{app:omitted}, we can choose smooth optimal potentials $(f, g)$ and $(f_n, g_n)$ so that the condition~\eqref{eq:dual-opt} holds for all $x,y \in \RR^d$.
Proposition~\ref{prop:pointwise_derivative_bound} shows that $f, f_n \in \mathcal F_{\tilde \sigma}$.

Strong duality implies that $S(P, Q) = \FA^{P,Q}(f,g)$ and $S(P_n, Q) = \FA^{P_n,Q}(f_n,g_n)$.
Moreover, by the optimality of $(f, g)$ and $(f_n, g_n)$ for their respective dual problems, we obtain
\begin{equation*}
\FA^{P,Q}(f_n,g_n)-\FA^{P_n,Q}(f_n,g_n) \leq \FA^{P,Q}(f,g)-\FA^{P_n,Q}(f_n,g_n) \leq \FA^{P,Q}(f,g)-\FA^{P_n,Q}(f,g)\,.
\end{equation*}
We conclude that
\begin{align*}
|S(P,Q) - S(P_n,Q)| & = |\FA^{P,Q}(f,g)-\FA^{P_n,Q}(f_n,g_n)| \\
& \leq  |\FA^{P,Q}(f,g)-\FA^{P_n,Q}(f,g)| + |\FA^{P,Q}(f_n,g_n)-\FA^{P_n,Q}(f_n,g_n)|\,.
\end{align*}
It therefore suffices to bound the differences $|\FA^{P,Q}(f,g)-\FA^{P_n,Q}(f,g)|$ and $|\FA^{P,Q}(f_n,g_n)-\FA^{P_n,Q}(f_n,g_n)|$.

Upon defining $h(x) :=  \int e^{g(y)-\frac{1}{2}||x-y||^2}\dd Q(y)$ we have
\begin{multline*}
\FA^{P,Q}(f,g)-\FA^{P_n,Q}(f,g) = \Big(\int f(x) (\dd P(x) - \dd P_n(x))\Big) \\ +  \Big(\int e^{f(x)} h(x) (\dd P(x) - \dd P_n(x)) \Big)\,.
\end{multline*}
Since $(f, g)$ satisfy $e^{f(x)}h(x) = 1$ for all $x \in \RR^d$, the second term above vanishes.
Therefore
\begin{equation*}
\begin{split}
|\FA^{P,Q}(f,g)-\FA^{P_n,Q}(f,g)| & = \Big|\int f(x) (\dd P(x) - \dd P_n(x))\Big| \\
& \leq \sup_{u \in \mathcal F_{\tilde \sigma}} \Big|\int u(x) (\dd P(x) - \dd P_n(x))\Big|\,.
\end{split}
\end{equation*}
Analogously,
\begin{equation*}
|\FA^{P,Q}(f_n,g_n)-\FA^{P_n,Q}(f_n,g_n)| \leq \sup_{u \in \mathcal F_{\tilde \sigma}} \Big|\int u(x) (\dd P(x) - \dd P_n(x))\Big|\,.
\end{equation*}
This proves the claim.
\end{proof}
Proposition~\ref{noexp} can be extended to apply to simultaneously varying $P_n$ and $Q_n$.
\begin{corollary}\label{cor:two_sample}
Let $P$, $Q$, $P_n$, and $Q_n$ be $\tilde \sigma^2$-subgaussian distributions, where $\tilde \sigma \in [0, \infty)$ is possibly random.
Then
\begin{multline*}
|S(P_n, Q_n) - S(P, Q)| \lesssim \sup_{u \in \mathcal F_{\tilde \sigma}} \Big|\int u(x) (\dd P(x) - \dd P_n(x))\Big| \\ + \sup_{v \in \mathcal F_{\tilde \sigma}} \Big|\int u(x) (\dd Q(x) - \dd Q_n(x))\Big|
\end{multline*}
almost surely.
\end{corollary}
\begin{proof}
By the triangle inequality,
\begin{equation}\label{eq:two-to-one}
|S(P_n, Q_n) - S(P, Q)| \leq |S(P, Q) - S(P_n, Q)| + |S(P_n, Q) - S(P_n, Q_n)|\,.
\end{equation}
Since $P$, $Q$, $P_n$, and $Q_n$ are all $\tilde \sigma^2$-subgaussian, Proposition~\ref{noexp} can be applied to both terms.
\end{proof}

The majority of our work goes into bounding the resulting empirical process.
Let $s \geq 2$.
Fix a constant $C_{s, d}$ and denote by $\mathcal{F}^s$ the set of functions satisfying
\begin{align}\label{eq:fs}
|f(x)| & \leq C_{s, d} (1 + \|x\|^2) \\
|D^\alpha f (x)| & \leq C_{s, d} (1+\|x\|^s) \quad \quad \forall \alpha: |\alpha| \leq s\,.
\end{align}
Proposition~\ref{prop:pointwise_derivative_bound} establishes that if $C_{s, d}$ is large enough, then $\frac{1}{1 + \sigma^{3s}} f \in \mathcal F^s$ for all $f \in \mathcal F_{\sigma}$.

The key result is the following covering number bound. Denote by $N(\varepsilon, \mathcal F^s, L_2(P_n))$ the covering number with respect to the (random) metric $L_2(P_n)$ defined by $\|f\|_{L_2(P_n)} = \left(\frac 1n \sum_{i=1}^n f(X_i)^2\right)^{1/2}$.

\begin{proposition}\label{covering}
Let $s = \lceil d/2 \rceil + 1$.
If $P$ is $\sigma^2$-subgaussian and $P_n$ is an empirical distribution, then there exists a random variable $L$ depending on the sample $X_1, \dots, X_n$ satisfying $E L \leq 2$ such that
\begin{equation*}
\log N(\varepsilon, \mathcal F^s, L_2(P_n)) \leq C_{d} L^{d/2s} \varepsilon^{-d/s} (1+\sigma^{2d})\,,
\end{equation*}
and
\begin{equation*}
\max_{f \in \mathcal F^s} \|f\|^2_{L_2(P_n)} \leq C_d (1+L \sigma^4)\,.
\end{equation*}
\end{proposition}

\begin{proof}
We use the symbol $C$, decorated with subscripts, to indicate constants whose value may change from line to line.
We apply \citet[Corollary 2.7.4]{VaaWel96}.
Denote by $L$ the quantity $\frac 1n \sum_{i=1}^n e^{\|x_i\|^2/2 d \sigma^2}$.
The subgaussianity of $P$ implies that $E L \leq 2$.
We partition $\RR^d$ into sets $B_j$ defined by $B_0 = [- \sigma, \sigma]^d$ and $B_j = [- 2^j \sigma, 2^j \sigma] \setminus [- 2^{j-1} \sigma, 2^{j-1} \sigma]$.
Note that for each $j$, the Lebesgue measure of $\{x : \mathrm{d}(x, B_j)\leq 1\}$ is bounded by $C_d(1 + \sigma^d 2^{dj})$.
Moreover, by Markov's inequality, the mass that $P_n$ assigns to each $B_j$ is at most $L e^{-2^{2j-3}}$.
Finally, by definition of the class $\mathcal{F}^s$, the functions in $\mathcal F^s$ have $\mathcal C^s(B_0)$ norm at most $C_{s,d} (1 + \sigma^{s})$, and on $B_j$ for $j \geq 1$ have $\mathcal C^s(B_j)$ norm at most $C_{s,d} 2^{js}(1 + \sigma^{s})$, where $\mathcal C^s(\Omega)$ represents the H\"older space on $\Omega$ of smoothness $s$.

Applying \citet[Corollary 2.7.4]{VaaWel96} with $V = d/s$ and $r = 2$ yields
\begin{align*}
\log N(\varepsilon, \mathcal F^s, L_2(P_n)) & \leq C_d \varepsilon^{-d/s} L^{d/2s} \left(\sum_{j \geq 0} (1+ \sigma^d 2^{dj})^{\frac{2s}{d+2s}}2^{\frac{2djs}{d+2s}} (1+\sigma^{s})^{\frac{2d}{d+2s}} e^{-  \frac{d2^{2j-3}}{d+2s}}\right)^{\frac{d+2s}{2s}} \\
& \leq  C_d \varepsilon^{-d/s} L^{d/2s} (1 + \sigma^{2d}) \left(\sum_{j \geq 0} 2^{\frac{4djs}{d+2s}} e^{-  \frac{d2^{2j-3}}{d+2s}} \right)^{\frac{d+2s}{2s}} \\
& \leq C_d \varepsilon^{-d/s} L^{d/2s}(1 + \sigma^{2d})\,,
\end{align*}
where the final step follows because the series is summable with value independent of $\sigma$ and $L$.

To show the second claim, we note that $E_{P_n} \|X\|^4 \leq C_d L \sigma^4$ by the same argument used to bound the moments of $P$ in Lemma~\ref{lem:subg_moments} in Appendix \ref{app:technical}.
The definition of the class $\mathcal F^s$ implies
\begin{equation*}
\max_{f \in \mathcal F^s} \|f\|^2_{L_2(P_n)} = \max_{f \in \mathcal F^s} E_{P_n} |f(X)|^2 \leq C_{d} E_{P_n} (1 + \|X\|^4) \leq C_d (1 + L \sigma^4)\,.
\end{equation*}
\end{proof}

We can now prove Theorem~\ref{theo1}.
\begin{proof}[Proof of Theorem~\ref{theo1}]
Let $\tilde \sigma$ be the infimum over all $\tau > 0$ such that $P$, $Q$, $P_n$, and $Q_n$ are all $\tau^2$-subgaussian.
By Lemma~\ref{lem:uniform_subg}, $\tilde \sigma$ is finite almost surely.

By Corollary~\ref{cor:two_sample},
\begin{align*}
E_{P, Q} |S(P, Q) - S(P_n, Q_n)|  & \lesssim E \sup_{u \in \mathcal F_{\tilde \sigma}} \Big|\int u(x) (\dd P(x) - \dd P_n(x))\Big| \\ & \phantom{\lesssim}+ E \sup_{v \in \mathcal F_{\tilde \sigma}} \Big|\int u(x) (\dd Q(x) - \dd Q_n(x))\Big|\,.
\end{align*}
We will show how to bound the first term, and the second will follow in exactly the same way.

For any set of functions $\mathcal F$, we write $\|P - P_n\|_{\mathcal F} = \sup_{u \in \mathcal F} (\int u(x) (\dd P(x) - \dd P_n(x)))$.
Recall that, for $s = \lceil d/2 \rceil + 1$, if $u \in \mathcal F_{\tilde \sigma}$ then $\frac{1}{1 + \tilde \sigma^{3s}} u \in \mathcal F^s$.
Therefore
\begin{align*}
E\|P - P_n\|_{\mathcal F_\sigma} & \leq E (1 + \tilde \sigma^{3s}) \|P - P_n\|_{\mathcal F^s} \\
& \leq (E(1 + \tilde \sigma^{3s})^2)^{1/2} (E \|P - P_n\|_{\mathcal F^s}^2)^{1/2}\,.
\end{align*}

Then by \citet[Theorem 3.5.1 and Exercise 2.3.1]{GinNic16}, we have
\begin{align*}
E \|P - P_n\|_{\mathcal F^s}^2 & \lesssim \frac 1 n E \left(\int_0^{\sqrt{\max_{f \in \mathcal F^s} \|f\|^2_{L_2(P_n)}}} \sqrt{\log 2 N(\tau, \mathcal F^s, L_2(P_n))} \, \dd \tau\right)^2 \\
& \leq C_d \frac 1 n E\left(\int_0^{C_d\sqrt{ (1+L \sigma^4)}} \sqrt{1 + L^{d/2s} \tau^{-d/s} (1+\sigma^{2d})} \, \dd \tau\right)^2 \\
& \leq C_d \frac 1 n (1 + \sigma^{2d}) E\left(\int_0^{C_d\sqrt{(1+L\sigma^4)}} L^{d/4s} \tau^{-d/2s} \, \dd \tau \right)^2 \\
& \leq C_d \frac 1 n (1 + \sigma^{2d}) E\left[(1+L\sigma^4)^{1- d/2s}\right]\,,
\end{align*}
where in the last step we have used that $d/2s < 1$ so that $\tau^{-d/2s}$ is integrable in a neighborhood of the origin.
Applying the bound on $E L$ yields that this expression is bounded by $C_d (1 + \sigma^{2d + 4})\frac 1 {n}$.

Lemma~\ref{lem:proxy_moments} in Appendix~\ref{app:technical} shows that $E \tilde \sigma^{2k} \leq C_k \sigma^{2k}$ for all positive integers $k$.
Combining these bounds yields
\begin{equation*}
E\|P - P_n\|_{\mathcal F_\sigma} \leq C_d (1 + \sigma^{3s})(1+\sigma^{d + 2}) \frac 1 {\sqrt n}\,,
\end{equation*}
as desired.
\end{proof}

\section{A central limit theorem for entropic OT}\label{sec:CLT}
The results of Section~\ref{sec:sample} show that, for general subgaussian measures, the empirical quantity $S(P_n, Q_n)$ converges to $S(P, Q)$ in expectation at the parametric rate.
However, in order to use entropic OT for rigorous statistical inference tasks, much finer control over the deviations of $S(P_n, Q_n)$ is needed, for instance in the form of asymptotic distributional limits.
In this section, we accomplish this goal by showing a central limit theorem (CLT) for $S(P_n, Q_n)$, valid for any subgaussian measures.

\citet{BigCazPap17} and \citet{KlaTamMun18} have shown CLTs for entropic OT when the measures lie in a \emph{finite} metric space (or, equivalently, when $P$ and $Q$ are finitely supported).
Apart from being restrictive in practice, these results do not shed much light on the general situation because OT on finite metric spaces behaves quite differently from OT on $\RR^d$.\footnote{A thorough discussion of the behavior of unregularized OT for finitely supported measures can be found in \citet{SomMun18} and \citet{WeeBac18}.}
Very recently, distributional limits for general measures possessing $4 + \delta$ moments have been obtained for unregularized OT by \citet{Del2019}.
Our proof follows their approach.

We prove the following.
\begin{theorem}\label{theo:clt}
Let $X_1, \dots X_n \sim P$ be an i.i.d.\ sequence, and denote by $P_n$ the corresponding empirical measure.
If $P$ is subgaussian, then 
\begin{equation}
\label{eq:cltone} \sqrt{n}\left(S(P_n,Q)-E(S(P_n, Q)\right) \overset{\mathcal{D}}{\rightarrow} \mathcal{N}\left(0, \var_P(f(X))\right),
\end{equation} 
and
\begin{equation}
\label{eq:varone} \lim_{n\rightarrow\infty} n \var(S(P_n,Q)) = \var_P(f(X))\,.
\end{equation}

Likewise, let $X_1,\ldots, X_n\sim P$ and $Y_1,\sim Y_m\sim Q$ are two i.i.d.\ sequences independent of each other. Assume $P$ and $Q$ are both subgaussian. Denote $\lambda:= \lim_{m,n\rightarrow\infty} \frac{n}{m+n}\in (0,1)$.

Then
\begin{equation}
\label{eq:clttwo}  \sqrt{\frac{mn}{m+n}}\left(S(P_n,Q_m)-E(S(P_n, Q_m)\right) \overset{\mathcal{D}}{\rightarrow} \mathcal{N}\left(0, (1-\lambda)\var_P(f(X_1)) + \lambda \var_Q(g(Y_1))\right),
\end{equation} 
and
\begin{equation}
\label{eq:vartwo} 
\lim_{m,n\rightarrow\infty}\frac{mn}{m+n} \var(S(P_n,Q_m)) = (1-\lambda)\var_P(f(X)) + \lambda \var_Q(g(Y)).
\end{equation}
\end{theorem}
The proof is deeply inspired by the method developed in \cite{Del2019} for the squared Wasserstein distance, and we roughly follow the same strategy.
\begin{proof}
The proof, in the one-sample case, proceeds as follows:
\begin{itemize}
\item[(a)] In Proposition~\ref{convergence} we show the optimal potentials for $(P_n,Q)$ convergence to optimal potentials for $(P, Q)$ uniformly on compact sets. 
\item[(b)] Letting $R_n := S(P_n, Q) - \int f(x)\dd P_n(x)$, we show in Proposition \ref{prop:bound}, that this uniform convergence implies that $\lim_{n\rightarrow \infty} n\var(R_n)=0$.
\item[(c)] The above convergence indicates $S(P_n,Q)$ can be approximated by the linear quantity $\int f(x)\dd P_n$. Then, \eqref{eq:cltone} and \eqref{eq:varone} are simply the limit statements (in distribution and $L^2$, respectively) applied to this linearization.
\end{itemize}
We omit the proof of the two-sample case as the changes to the argument (see Theorem 3.3. in \cite{Del2019}, for the squared Wasserstein distance) adapt in a straightforward way to the entropic case. 
\end{proof}

We finish this section with the statement and proof of Proposition~\ref{convergence}, which may be of independent interest. We defer to Appendix \ref{app:omitted} the statement and proof of Proposition \ref{prop:bound} since many of the arguments have been presented in \cite{Del2019}.

\begin{proposition}\label{convergence}
Let $P_n,Q_n$ be empirical measures, $P$ and $Q$ both assumed subgaussian. 
There exist $(f_n,g_n)$ optimal potentials for $(P_n,Q_n)$ such that $(f_n,g_n)$ converges uniformly in compacts to optimal potentials $(f,g)$ for $P$ and $Q$.
\end{proposition}
\begin{proof}

 The proof is inspired by \cite{Feydy2018} and we divide it in two steps:
 \begin{itemize}\item[Step 1]  By using the following extended version of the Arzela-Ascoli theorem we find a convergent subsequence: suppose $h_n$ is a sequence of functions in $\mathbb{R}^d$ satisfying
\begin{itemize}
\item[(a)] Local equicontinuity: for each $x_0\in\mathbb{R}^d$ and $\epsilon>0$, there is a $\delta>0$ such that
\begin{equation} \nonumber
\label{eq:equi}
||x-x_0||<\delta \quad \text{implies} \quad |h_n(x)-h_n(x_0)|<\epsilon \quad \text{for all } n
\end{equation}
\item[(b)]Pointwise boundedness: for each $x$, the sequence $h_n(x)$ is bounded.
\end{itemize}
Then, there exist a subsequence $h_{n_j}$ that converges uniformly on compacts to a continuous function $h$.
\item[Step 2] We prove the limit functions are optimal for $(P,Q)$ and conclude the entire sequence converges by a uniqueness argument.
\end{itemize}
\textit{Proof of Step 1}:
By Lemma~\ref{lem:uniform_subg} in Appendix \ref{app:technical}, there exists a (random) $\sigma^2$ such that the measures $\{P_n\}$ are uniformly $\sigma^2$-subgaussian.
We choose $(f_n, g_n)$ and $(f, g)$ as in Proposition~\ref{extend-potentials} in Appendix \ref{app:omitted}

By Proposition~\ref{extend-potentials} in Appendix \ref{app:omitted}, $(f_n, g_n)$ are pointwise bounded by a quantity independent of $n$.
Likewise, Proposition~\ref{prop:pointwise_derivative_bound} implies that the derivatives of $f_n$ and $g_n$ are also pointwise bounded, which implies local equicontinuity.

We conclude for a certain subsequence $n_j$, $(f_{n_j},g_{n_j})$  converges to some $(f_\infty,g_\infty)$.

\textit{Proof of Step 2:}
It is easy to verify (by Jensen's inequality and dominated convergence) that Proposition 11 in \cite{Feydy2018}, holds in arbitrary domains (not necessarily bounded), and we can assume $(f,g)$ are unique $(P\otimes Q)$-a.s. once we fix $E_P f(X) = E_Q g(Y)$.
Notice that if $f_\infty=f,g_\infty=g$, $P$-a.s. and $Q$-a.s. we can conclude: on each compact we apply the above argument starting with any arbitrary subsequence $n_k$ and find a subsequence such that $f_{n_{k_j}}\rightarrow f, g_{n_{k_j}}\rightarrow g$; therefore $f=\lim f_n(x)$ and $g(y)=\lim g_n$, uniformly in compacts.

It therefore suffices to show that that i) $(f_\infty, g_\infty)$ satisfy the dual optimality conditions and that $f_\infty$ (respectively $g_\infty$) is $P$ (respectively $Q$) integrable, with $E_P f_\infty(X) = E_Q g_\infty(Y)$.
Let's prove i.
Passing to a subsequence, we assume $f_n \to f$ and $g_n \to g$ uniformly on compact sets.
We have
\begin{align*}
e^{-f_\infty(x)} & = \lim_{n \to \infty} \int e^{g_{n}(y) -  \frac 1 2 \|x - y\|^2} \, \dd Q_{n}(y) \\
e^{-g_\infty(y)} & = \lim_{n \to \infty} \int e^{f_{n}(x) -  \frac 1 2 \|x - y\|^2} \, \dd P_{n}(x)\,.
\end{align*}
It suffices to show that the order of the limit and integral on the right side can be swapped.
For a fixed $x$ we observe that Proposition~\ref{extend-potentials} implies that the integrand is dominated by a uniformly integrable function.
Therefore for an arbitrary $\varepsilon > 0$ there exists a compact set $K$ such that
\begin{align*}
\int_{K^C}  e^{g_{\infty}(y) -  \frac 1 2 \|x - y\|^2} \, \dd Q(y) & \leq \varepsilon \\
\int_{K^C}  e^{g_{n}(y) -  \frac 1 2 \|x - y\|^2} \, \dd Q_{n}(y) & \leq \varepsilon \quad \quad \forall n \geq 0\,.
\end{align*}

Write $v_n(y) = e^{g_{n}(y) -  \frac 1 2 \|x - y\|^2}$ and $v_\infty = e^{g_{\infty}(y) -  \frac 1 2 \|x - y\|^2}$.
Since $g_{n}$ converges uniformly in compacts so does $v_n$; in particular, there exists $n_0$ such that if $n\geq n_0$, \begin{equation}\label{eq:unifK}
\left|v_n(y)-v_\infty(y)\right|\leq \epsilon, \forall y\in K.
\end{equation}
Also, since $v_\infty$ is $Q$-integrable, by the strong law of large numbers, almost surely there exists an $n_1$ such that if $n\geq n_1$,
\begin{equation}\label{eq:weakc}
\left|\int v_\infty(y)\dd Q_n(y)-\int v_\infty(y)\, \dd Q(y)\right| \leq \epsilon,
\end{equation}
We obtain that for $n$ sufficiently large,
\begin{equation*}
\left|\int v_n(y)\dd Q_n(y)- \int v_\infty(y) \, \dd Q(y)\right| \leq 4 \varepsilon\,.
\end{equation*}
Since $\varepsilon$ was arbitrary, we obtain
\begin{equation*}
e^{-f_\infty(x)} = \int v_\infty(y) \, \dd Q(y) = \int e^{g_{\infty}(y) -  \frac 1 2 \|x - y\|^2} \, \dd Q_{n}(y)\,.
\end{equation*}
Repeating the proof for $g_\infty$, we obtain that $(f_\infty, g_\infty)$ satisfy the dual optimality conditions.

Clearly $(f_\infty, g_\infty)$ are integrable by dominated convergence, and an argument analogous to the one used to show dual optimality establishes that $E_P f_\infty(X) = E_Q g_\infty(Y)$.
The claim is therefore proved.
 \end{proof}

\section{Application to entropy estimation}\label{sec:entropy}
In this section, we give an application of entropic OT to the problem of entropy estimation.
First, in Proposition \ref{propent} we establish a new relation between entropic OT and the differential entropy of the convolution of two measures.
Then, as a corollary of this and the previous sections results we prove Theorem \ref{theo:3}, stating that entropic OT provides us with a novel estimator for the differential entropy of the (independent) sum of a subgaussian random variable and a gaussian random variable, and for which performance guarantees are available.

Throughout this section $\nu$ denotes a translation invariant measure. Whenever $P$ has a density $p$ with respect to $\nu$, we define its $\nu$-differential entropy as  
$h(P):=-\int p(x)\log p(x) d\nu(x)= -H(P\lvert \nu).$

The following proposition links the differential entropy of a convolution with the entropic cost.
\begin{proposition}\label{propent}

Let $\Phi_g$ be the measure with $\nu$-density $\phi_g(y)=Z_g^{-1}e^{-g(y)}$ for a smooth $g$, and define $Q=P\ast \Phi_g$, with $P\in\mathcal{P}(\mathbb{R}^d)$ arbitrary. The $\nu$-density of $Q$, $q(y)$, satisfies
\begin{equation}\nonumber q(y) = \int \phi_g(y-x)\dd P(x) = \int Z_g^{-1}e^{-g(y-x)}\dd P(x).
\end{equation}
Consider the cost function $c(x,y):=g(x-y)$ (not necessarily quadratic). Then, the optimal entropic transport cost and differential entropy are linked through
\begin{equation}
 \label{eq:I} h(P\ast \Phi_g)= S(P,P\ast \Phi_g)  + \log(Z_g).
\end{equation}

\end{proposition}

\begin{proof}

Define a more general entropic transportation cost involving the generic $c$ and probability measures $\alpha,\beta$:
\begin{equation}\label{eq:sink}S^{\alpha\otimes\beta}(P,Q) := \inf_{\pi\in \Pi(P,Q)} \left[\int c(x,y)d\pi(x,y) + H(\pi\lvert \alpha\otimes\beta)\right].\end{equation}
Observe we may re-write \eqref{eq:sink} as
\begin{eqnarray}\label{eq:sink2}\nonumber S^{\alpha\otimes\beta}(P,Q) &=& \inf_{\pi\in \Pi(P,Q)} \left[\int_{\mx\times\my} c(x,y)d\pi(x,y) + H(\pi\lvert P\otimes Q)\right] +H(P\otimes Q \lvert \alpha\otimes \beta) \\  & = & S(P,Q) +H(P\otimes Q \lvert \alpha\otimes \beta) . \end{eqnarray}

Additionally, it can be verified an alternative representation for \eqref{eq:sink} is the following
\begin{equation}\label{eq:sink1}S^{\alpha\otimes\beta}(P,Q) = \inf_{\pi\in \Pi(P,Q)}H\left(\pi \bigg\lvert Z^{-1}e^{-c}\alpha\otimes\beta\right) -\log(Z),\end{equation}
where $Z$ is the number making $\Lambda:= Z^{-1}e^{-c}\alpha\otimes\beta$ a \textit{bona fide} probability measure. 

Now, take $\alpha=P$, $\beta=\nu$ and $Q=P\ast \Phi_g$ in the above expressions. For these choices we have $Z=Z_g$. Indeed, by the translation invariance of $\nu$, we have
\begin{eqnarray}\nonumber Z = \iint e^{-c(x,y)}\dd P(x)d\nu(y)&=& \int \left(\int e^{-g(y-x)}d\nu(y)\right)\dd P(x) \\ \nonumber &=&  \int \left(\int e^{-g(y)}d\nu(y)\right)\dd P(x) \\ \nonumber & =& \int Z_g \dd P(x) =Z_g.  \end{eqnarray}
Then, $d\Lambda(x,y)=\dd P(x)\phi_g(y-x)d\nu(y)$, and by marginalization we deduce $\Lambda\in\Pi(P,P\ast \Phi_g)$. Therefore, the right side of \eqref{eq:sink1} equals $H(\Lambda\lvert \Lambda) - \log Z_g= -\log Z_g$. Finally, we combine \eqref{eq:sink2} and \eqref{eq:sink1} to obtain
\begin{equation*}
\label{rel1}-\log Z_g = S(P,P\ast \Phi_g) + H\left(P\otimes \left(P\ast \Phi_g\right) \lvert P \otimes \nu\right),
\end{equation*}
and achieve the final conclusion after noting that
\begin{equation*}
H(P\otimes \left(P\ast \Phi_g\right) \lvert P\otimes \nu) = H(P\lvert P)+H\left(P\ast \Phi_g \vert \nu\right)=H\left(P\ast \Phi_g\lvert \nu\right)=-h(P\ast\Phi_g).
\end{equation*}
\end{proof}
Now we can state the following theorem.
\begin{theorem}\label{theo:3}
Let $P$ be subgaussian, $G\sim\mathcal{N}(0,\sigma_g^2 I_d)$. Denote $Q=P\ast \Phi_g$ the distribution of the sum of an independent $X\sim P$ and $G$, and define the plug in estimator $\hat h(Q) = S(P_n, Q_m) +\log Z_g$ where $P_n$ and $Q_m$ are independent samples from $P$ and $Q$.
Then,
\begin{itemize}
\item[(a)]If $m=n$, $$ \sup_{P}E_{P}|\hat{h}(Q) - h(Q)| \leq  O\left(\frac{1}{\sqrt{n}} \right).$$
\item[(b)] The limit
\begin{equation}
\sqrt{\frac{mn}{m+n}}\left(\hat{h}(Q) - E(\hat{h}(Q)\right) \overset{\mathcal{D}}{\rightarrow} \mathcal{N}\left(0,  \lambda \var_Q(\log q(Y))\right)
\end{equation}
holds, where  $\lambda = \lim_{m, n \to \infty} \frac{n}{m+n}$. Moreover, $\lim_{m,n\rightarrow\infty}\frac{mn}{m+n}\var(\hat{h}(Q))=\lambda \var_Q(\log q(Y))$.
\end{itemize}
\end{theorem}
\begin{proof}
(a) is a simple re-statement of Theorem \ref{theo1} in the light of Proposition \ref{propent}. (b) is a re-statement of Theorem \ref{theo:clt}, after noting in this case the optimal potentials are $(f,g)=(-\log Z_g, -\log q)$.
\end{proof}
The rate $1/\sqrt n$ in Theorem~\ref{theo:3} is also achieved by a different estimator proposed by~\citet{Goldfeld2018b} (see also \citealp{Weed2018b}), but this estimator lacks distributional limits.

\section{Empirical results}\label{sec:experiments}
\begin{figure}\includegraphics[width=1.0\textwidth]{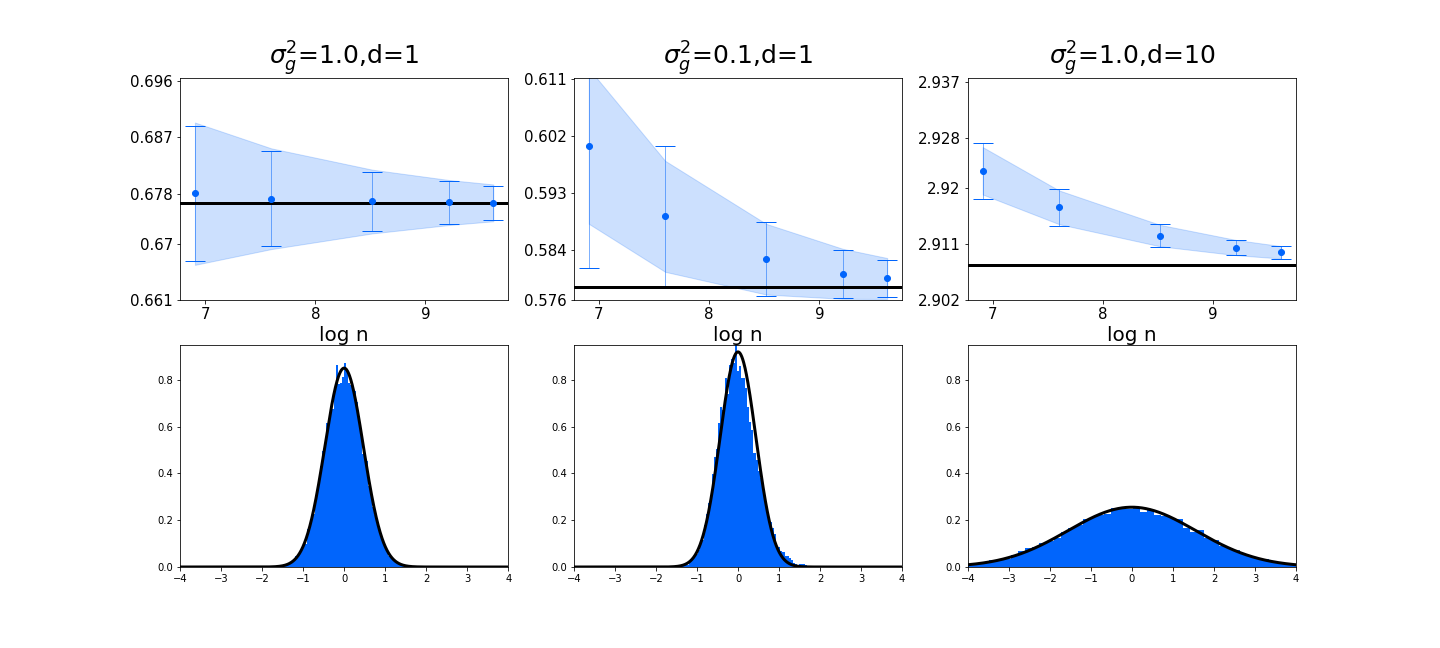}
\caption{Top row: $ES(P_n,Q_n)$ as a function of $n\in\{1000,2000,50000,10000,15000\}$, computed from $16,000$ repetitions for each value of $n$. The shading corresponds to one standard deviation of \mbox{$S(P_n,Q_n)-ES(P_n,Q_n)$}, assuming the asymptotics of Theorem \ref{theo:clt} are valid. Error bars are one sample standard deviation long on each side. Both $x$ and $y$ axes are in logarithmic scale. Bottom row: histograms of $\sqrt{\frac{nn}{n+n}}\left(S(P_n,Q_n)-ES(P_n,Q_n))\right)$ when $n=15000$. Ground truth (numerical integration) is shown with black solid lines.}
\label{fig:clt}
\end{figure}
We provide empirical evidence supporting and illustrating our theoretical findings. We focus on the entropy estimation problem because there are closed form expressions for the potentials (see Theorem \ref{theo:3}), and because it allows a comparison with the estimator studied in \citep{Goldfeld2018b}.

Specifically, consider $X\sim P=\frac{1}{2}\left(\mathcal{N}(1_d,I_d)+\mathcal{N}(-1_d,I_d)\right)$, the mixture of the gaussians centered at $1_d:=(1,\ldots,1)$ and $-1_d$.  We aim to estimate the entropy of the new mixture $Q=P\ast\Phi_g$.

Figure \ref{fig:clt}, top, shows the convergence of $E S(P_n, Q_n)$ to $S(P, Q)$.
Consistent with the bound in Theorem~\ref{theo1} and Corollary~\ref{cor:varying-ep}, $S(P_n, Q_n)$ is a worse estimator for $S(P, Q)$ when $d$ is large or the regularization parameter is small. We also plot the predicted (shading) and actual (bars) fluctuations of $S(P_n, Q_n)$ around its mean. Though the CLT holds only in the asymptotic limit, these experiments reveal that the empirical fluctuations in the finite-$n$ regime are broadly consistent with the predictions of the CLT.
Figure \ref{fig:clt}, bottom, shows that the empirical distribution of the rescaled fluctuations is an excellent match for the predicted normal distribution.

In Figure \ref{fig:means} we compare the performance between entropic OT-based estimators from Theorem \ref{theo:3} and $\hat{h}_{\text{m.g.}}(Q)$, the one from \citep{Goldfeld2018b}, where $h(P\ast \Phi_g)$ is estimated as the entropy of the mixture of gaussians $P_n\ast \Phi_g$, in turn approximated by Monte Carlo integration. We consider two OT-based estimators, $\hat{h}_{\text{ind}}(Q)$ where $P_n,Q_n$ are completely independent (i.e., the one used for Figure \ref{fig:clt}), and $\hat{h}_{\text{paired}}(Q)$ where samples $Q_n$ are drawn by adding gaussian noise to $P_n$. Observe that our sample complexity and CLT results are only available for $\hat{h}_{\text{ind}}(Q)$.

Results show a clear pattern of dominance, with $E\hat{h}_{\text{paired}}(Q)$ achieving the fastest convergence. The main caveat is the extra memory cost: while $\hat{h}_{\text{m.g.}}(Q)$ can be computed sequentially with each operation requiring $O(n)$ memory, in the most naive implementation (used here) both $\hat{h}_{\text{paired}}(Q), \hat{h}_{\text{ind}}(Q)$ demand $O(n^2)$ space for storing the matrix $D_{i,j}=e^{-||x_i-y_j||^2/{2\sigma_g^2}}$, to which the Sinkhorn algorithm is applied. This memory requirement might be alleviated with the use of stochastic methods \citep{Genevay2016,Bercu2018}. 

We leave for future work both the implementation of more scalable methods for entropic OT, and a detailed theoretical analysis of different entropic OT-based estimators (e.g. $\hat{h}_{\text{paired}}(Q)$ v.s. $\hat{h}_{\text{ind}}(Q)$) that may bring about a better understanding of their observed substantial differences.

\begin{figure}[ht]\includegraphics[width=1.0\textwidth]{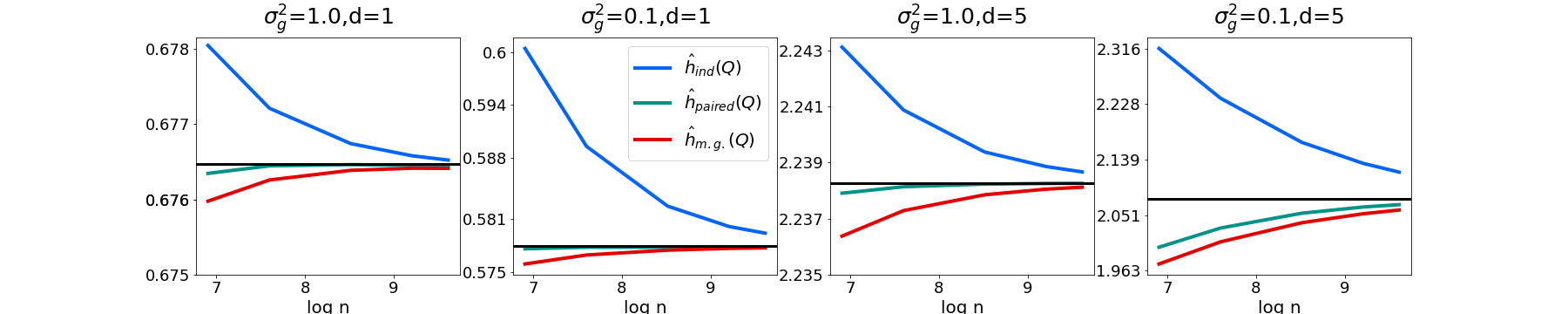}
\caption{Comparison between $E\hat{h}_{\text{ind}}(Q),E\hat{h}_{\text{paired}}(Q),E\hat{h}_{\text{m.g.}}(Q)$. Details are the same as in Figure \ref{fig:clt}.}
\label{fig:means}
\end{figure}

\appendix 
\section{Omitted Results and Proofs}\label{app:omitted}

\begin{proposition}\label{extend-potentials}
Let $P$ and $Q$ be two $\sigma^2$-subgaussian distributions.
Then there exist smooth optimal potentials $(f, g)$ for $S(P, Q)$ such that
\begin{alignat*}{2}
-d \sigma^2(1+ \frac 12 (\|x\| + \sqrt{2 d}\sigma)^2) - 1 \mskip\thickmuskip&& \leq f(x) &\leq \frac 12 (\|x\| + \sqrt{2 d}\sigma)^2\\
-d \sigma^2(1+ \frac 12 (\|y\| + \sqrt{2 d}\sigma)^2) - 1 \mskip\thickmuskip&& \leq g(y) & \leq \frac 12 (\|y\| + \sqrt{2 d}\sigma)^2
\end{alignat*}
and the dual optimality conditions~\eqref{eq:dual-opt} hold for \emph{all} $x, y \in \RR^d$.
\end{proposition}

\begin{proof}[Proof of Proposition \ref{extend-potentials}]
Let $(f_0, g_0)$ be any pair of optimal potentials.
Since $(f_0 + K, g_0 - K)$ also satisfy the optimality conditions and $f_0 \in L_1(P)$ and $g_0 \in L_1(Q)$,  we can assume without loss of generality that $E_P f_0(X) = E_Q g_0(Y) = \frac 12 S(P, Q) \geq 0$.
We define
\begin{align*}
f(x) & = - \log \int e^{g_0(y) - \frac 12 \|x - y\|^2} \, \dd Q(y) \\
g(y) & = - \log \int e^{f(x) - \frac 1 2 \|x - y\|^2} \, \dd P(x)\,,
\end{align*}
for all $x, y \in \RR^d$.

We need to  check that these integrals are well defined.
First, Jensen's inequality implies
\begin{align*}
g_0(y) & = - \log \int e^{f_0(x) - \frac 12 \|x - y\|^2} \, \dd P(x) \\
& \leq - E_P f_0(X) + \frac 12 E_P \|X - y\|^2 \\
& \leq \frac 12 E_P \|X - y\|^2
\end{align*}
for $Q$-a.e.\ y.
Therefore
\begin{equation*}
e^{g_0(y) - \frac 12 \|x - y\|^2} \leq e^{\frac 1 2 E_P \|X - y\|^2 - \frac 1 2 \|x - y\|^2}
\end{equation*}
for $Q$-a.e.\ y.
By Lemma~\ref{lem:subg_moments} in Appendix \ref{app:technical}, $E_P \|X\|^2 \leq 2 d \sigma^2$, which implies that $e^{g_0(y) - \frac 12 \|x - y\|^2}$ is dominated by $e^{d \sigma^2 + (\|x\| + \sqrt{2} d \sigma) \|y\|}$.
Subgaussianity implies
\begin{equation*}
\int e^{d \sigma^2 + (\|x\| + \sqrt{2} d \sigma) \|y\|} \, \dd Q(y) \leq 2 e^{d \sigma^2(1+ \frac 12 (\|x\| + \sqrt{2 d}\sigma)^2} < \infty
\end{equation*}
Therefore $f(x)$ is well defined for all $x \in \RR^d$.
The same argument used to bound $g_0$ holds for $f$ as well, which implies that $g$ is also well defined.
Therefore our definitions of $f$ and $g$ are valid on the whole space, and moreover the claimed lower bounds on $f$ and $g$ hold.
Jensen's inequality combined with the inequalities $E_Q g_0(Y) \geq 0$ and $E_P f(X) \geq 0$ yield the upper bounds.
The smoothness of $f$ and $g$ follows from an easy application of dominated convergence.

We now show that $(f, g)$ are optimal potentials.
By construction $\int e^{f(x) + g(y) - \frac{1}{2}||x- y||^2} \, \dd P(x) = 1$ for all $y \in \RR^d$.
Now, note that
\begin{align*}
\int e^{f(x) + g(y) - \frac{1}{2}||x- y||^2} \, \dd P(x) \dd Q(y) & = \int e^{f(x) + g_0(y) - \frac{1}{2}||x- y||^2} \, \dd P(x) \dd Q(y) \\
& = \int e^{f_0(x) + g_0(y) - \frac{1}{2}||x- y||^2} \, \dd P(x) \dd Q(y)\,.
\end{align*}
Jensen's inequality yields
\begin{align*}
\int \left(f - f_0\right)(x) \, \dd P(x) + \int (g - g_0)(y) \, \dd Q(y) & \geq - \log \int e^{f_0(x) - f(x)} \, \dd P(x) - \log \int e^{g_0(y) - g(y)} \, \dd Q(y) \\
& = - \log \int e^{f_0(x) + g_0(y) - \frac{1}{2}||x- y||^2} \, \dd P(x) \dd Q(y) \\ &\quad- \log \int e^{f(x) + g_0(y) - \frac{1}{2}||x- y||^2} \, \dd P(x) \dd Q(y) \\
& = 0\,.
\end{align*}
Since $(f_0, g_0)$ maximizes~\eqref{eq:dual}, so does $(f, g)$. Therefore $(f, g)$ are optimal potentials.
In particular, this implies that $\int \left(g - g_0\right)(y) \, \dd Q(y) = \log \int e^{g_0(y) - g(y)} \, \dd Q(y)$, and hence \mbox{$g = g_0$} $Q$-almost surely by the strict concavity of the logarithm function.
We obtain that \mbox{$\int e^{f(x) + g(y) - \frac{1}{2}||x- y||^2} \, \dd Q(y) = \int e^{f(x) + g_0(y) - \frac{1}{2}||x- y||^2} \, \dd Q(y) = 1$} for all $x \in \RR^d$.
\end{proof}

\begin{proof}[Proof of proposition \ref{prop:pointwise_derivative_bound}]
We choose the potentials $f$ and $g$ as in Proposition~\ref{extend-potentials}.
That establishes the $k = 0$ case.

For convenience, write $\overline f(x) = f(x) - \frac 12 \|x \|^2$.
We seek to bound $|D^\alpha \overline f(x)|$.

Our calculation is similar to classical calculations which relate the cumulants of a distribution to its moments \citep[see][Section 2.3]{McC87}.
Given a multi-index $\beta$, write
\begin{equation}\label{eq:generalized_moments}
\mu_\beta = \frac{\int y^{\beta} e^{g(y) - \frac 12 \|y\|^2 + x \cdot y} \, \dd Q(y)}{\int e^{g(y) - \frac 12 \|y\|^2 + x \cdot y} \, \dd Q(y)}\,.
\end{equation}
We use the convention that $y^\beta = \prod_{i=1}^d y_i^{\beta_i}$.
The notation $\mu_\beta$ is chosen to remind the reader that these quantities are  moments of $y$ under the tilted measure whose density with respect to $Q$ is proportional to $e^{g(y) - \frac 12 \|y\|^2 + x \cdot y}$.

By the multivariate Fa\'a di Bruno formula \citep[see, e.g.][]{ConSav96},
\begin{equation}\label{eq:faa-di-bruno}
D^\alpha \overline f(x) = - D^\alpha \log(e^{- \overline f(x)}) = \sum_{\substack{\beta_1, \dots \beta_k \\ \beta_1 + \dots + \beta_k = \alpha}} \lambda_{\alpha, \beta_1, \dots, \beta_k} \prod_{j=1}^k \mu_{\beta_j}\,,
\end{equation}
where the coefficients $\lambda_{\alpha, \beta_1, \dots, \beta_k}$ are combinatorial quantities related to partitions of $[k]$ whose precise value is unimportant.

Applying Lemma~\ref{lem:moment_bound} in Appendix \ref{app:technical} yields the claim.
\end{proof}

\begin{proposition}

\label{prop:bound}
Assume $P$ and $Q$ are subgaussian.
Let $(f,g)$ be the corresponding optimal dual potentials constructed in Proposition~\ref{extend-potentials}, and define
\begin{equation}\nonumber
R_n = S(P_n,Q)-\int f(x)\dd P_n(x).
\end{equation}

Then,
\begin{equation*}
\lim_{n\rightarrow \infty} n \var(R_n)=0.
\end{equation*}

\end{proposition}

Our proof relies on the tensorization property for the variance \citep{efron1981jackknife,boucheron2013concentration,van2014probability}, also known as Efron-Stein inequality: Let $X_1,\ldots X_n$ be i.i.d r.v's with distribution $P$ and $X'_1,\ldots X'_n$ be independent copies of $X_1,\ldots X_n$. Also, let $w$ be an arbitrary measurable function of the sample that is symmetric on its coordinates, and define $Z=\omega\left(X_1,\ldots X_n\right)$ and $Z'=\omega\left(X'_1, X_2, \ldots X_n\right)$. Then,
\begin{equation}
\label{eq:efstein}
Var(Z)\leq \frac{n}{2}E(Z-Z')_+^2.\end{equation}
\begin{proof}[Proof of Proposition~\ref{prop:bound}]
Denote by $P'_n$ the empirical distribution of $X'_1,X_2,\ldots X_n$,
and let
\begin{equation}\nonumber
R'_n = S(P'_n,Q)-\int f(x)\dd P'_n(x).
\end{equation}
by Efron-Stein, it suffices to show $\lim_{n\rightarrow \infty} n^2E(R_n-R'_n)^2_+=0$. We divide the proof in the verification of two statements.
First, we show $\lim_{n\rightarrow \infty}n(R_n-R'_n)_+=0$.
We will then show that $n^2 (R_n - R'_n)_+^2$ is uniformly integrable.

Call $(f_n,g_n)$ the optimal potentials associated to $(P_n,Q)$.
Since $P_n$ is subgaussian by Lemma~\ref{lem:uniform_subg} in Appendix \ref{app:technical}, Proposition~\ref{extend-potentials} implies that we can assume that $(f_n, g_n)$ satisfy the dual optimality conditions for all $x, y \in \RR^d$.
Therefore
\begin{align*}
S(P_n,Q) & = \int f_n(x) \, \dd P_n(x) + \int g_n(y)\, \dd Q(y), \\
S(P'_n,Q) &\geq \int f_n(x)\dd P'_n(x) + \int g_n(y)\dd Q(y) -\iint e^{f_n(x)+g_n(y)-\frac{1}{2}||x-y||^2}\dd P'_n(x)\dd Q(y) + 1 \\
& = \int f_n(x)\, \dd P'_n(x) + \int g_n(y) \, \dd Q(y)\,.
\end{align*}
Therefore,
\begin{equation*}
n(R_n-R'_n)_+ \leq (f_n(X_1)-f(X_1)) -  (f_n(X'_1)-f(X'_1))\,.
\end{equation*}
By Proposition \ref{convergence}, ($f_n,g_n)$ converges pointwise to $(f,g)$ almost surely, so $\lim_{n \to \infty} n(R_n - R_n')_+ = 0$ almost surely.

To show uniform integrability, we note that $n(R_n-R'_n)=n(S(P_n,Q)-S(P_n,Q)) - (f(X_1)-f(X'_1))$ and by Proposition~\ref{extend-potentials} and the subgaussianity of $P$, $f(X_1),f(X'_1)$ have finite second moments.
It therefore suffices to show that $n^2(S(P_n,Q) -S(P'_n,Q))^2_+$ is uniformly integrable.  

Let $\pi'$ be the underlying optimal entropic coupling between $P'_n$ and $Q$ that we disintegrate in terms of $Q$ and the (random) kernel $\{P'(\cdot|y)\}_y$ of conditional distributions over the sample $P'_n$ given $y$, i.e.

\begin{equation*}
\dd\pi'(x,y)=\dd Q(y)\left(P'(x|y)\delta_{X'_1}(x) +\sum_{i=2}^n P'(x|y) \delta_{X_i}(x)\right)\,.
\end{equation*}

We now slightly modify $\pi'$ to make it have $P_n$ as first marginal; specifically, we define
\begin{equation}\nonumber
d\bar{\pi}(x,y)=\dd Q(y)\left(\sum_{i=1}^n \bar{P}(x|y) \delta_{X_i}(x)\right), \text{ with } \bar{P}(x|y)=\begin{cases} P'(X'_1|y) & x=X_1 \\ P'(X_i|y) & x=X_i,i\neq 1\end{cases}.
\end{equation}
By the definitions of $S(P_n, Q)$ and $S(P'_n, Q)$, it is easily verified that
\begin{equation}\nonumber
S(P_n, Q) \leq   \sum_{i=1}^n \int \frac{\norm{X_i-y}^2}{2} \bar{P}(X_i|y)\dd Q(y) +I(\bar{\pi}),
\end{equation}
and that
\begin{equation}\nonumber
S(P_n', Q) =    \int \frac{\norm{X'_1-y}^2}{2} P'(X'_1|y)\dd Q(y)+ \sum_{i=2}^n \int \frac{\norm{X_i-y}^2}{2} P'(X_i|y)\dd Q(y) +I(\pi')\,,
\end{equation}
where $I(\cdot)$ denotes mutual information.
Therefore, 
\begin{equation}\label{eq:entropydif}
S(P_n, Q)-S(P_n', Q) \leq I(\bar{\pi})-I(\pi')  + \int \frac{\norm{X_1-y}^2-\norm{X'_1-y}^2}{2} P'(X'_1|y)\dd Q(y). 
\end{equation}
Observe that $I(\bar{\pi})=I(\pi')$ since $I(\pi')$ doesn't depend on the sample values, but only in the way the conditionals $P'(\cdot|y)$ split over the sample, which by construction is the same for both $\bar{\pi}$ and $\pi'$. Therefore, we only need to bound the (expected squared) integral in \eqref{eq:entropydif}, and we proceed as in \cite{Del2019}. Specifically, we have
\begin{eqnarray*}
\label{eq:delbarriobound}S(P_n, Q)-S(P_n', Q)&\leq& \int \frac{\norm{X_1-y}^2-\norm{X'_1-y}^2}{2} P'(X'_1|y)\dd Q(y)\\ \nonumber &\leq &\frac{1}{2}\norm{X_1-X'_1}\left(\frac{\norm{X_1} +\norm{X'_1}}{n} +2\int \norm{y} P'(X'_1|y)\dd Q(y)\right),
 \end{eqnarray*}
from which it follows that
\begin{equation}
 \label{eq:expzpos}
 n^{2 } (S(P_n, Q)-S(P_n', Q))^{2}_+\leq (\norm{X_1-X'_1}^2\norm{X_1}^2)+n^2\norm{X_1-X'_1}^2\left(\int \norm{y} P'(X'_1|y)\dd Q(y)\right)^2.
\end{equation}
The first term is clearly uniformly integrable since $P$ has moments of all orders, so we focus on the second term.

By Cauchy-Schwartz,  
\begin{eqnarray}\label{eq:cz}
\nonumber E\left(\norm{X_1-X'_1}^{4}\left(\int \norm{y} P'(X'_1|y)\dd Q(y)\right)^{4}\right)^2 & \leq &E\left(\norm{X_1-X'_1}^{8}\right)\times \\  &&E\left(\left(\int \norm{y} P'(X'_1|y)\dd Q(y)\right)^8\right).\end{eqnarray}

And now, by H\"{o}lder's inequality\,,
\begin{eqnarray*}\nonumber 
\label{eq:boundnormy}\left(\int \norm{y} P'(X'_1|y)\dd Q(y)\right)^{8} &\leq & \left(\int P'(X'_1|y)\dd Q(y)\right)^{7}\left(\int \norm{y}^{8} P'(X'_1|y)\dd Q(y)\right) \\
& = &\frac{1}{n^{7}} \left(\int \norm{y}^{8} P'(X'_1|y)\dd Q(y)\right).
\end{eqnarray*}

Also, notice that the r.v's $\int \norm{y}^8 P'(X'_i|y)\dd Q(y)$ are equally distributed, and therefore
\begin{align*}
E\left(\int \norm{y}^8 P'(X'_1|y)\dd Q(y)\right)=\frac{1}{n}E\left(\sum_{i=1}^n \int \norm{y}^8 P'(X'_i|y)\dd Q(y)\right) =\frac{1}{n}E\left(\int \norm{y}^8 \dd Q(y)\right).
\end{align*}

We obtain
\begin{equation}
\label{eq:boundenormy2} 
E\left(\left(\int \norm{y} P'(X'_1|y)\dd Q(y)\right)^8\right) \leq \frac{1}{n^8} \int \norm{y}^8\dd Q(y).
\end{equation}

Together, \eqref{eq:cz}  and\eqref{eq:boundenormy2} imply that the quantity $n^2 \norm{X_1 - X_1'}^2 \left(\int \norm{y} P'(X'_1|y)\dd Q(y)\right)^2$ has uniformly bounded second moments, and is therefore uniformly integrable.
Therefore \mbox{$n^2(S(P_n, Q) - S(P_n', Q))_+^2$} is uniformly integrable as well, and combining this with the almost sure convergence implies the claim.
\end{proof}

\section{Technical Lemmas}\label{app:technical}
Throughout this appendix, the symbol $C$ will be used to indicate an unspecified positive constant whose value may change from line to line.
Subscripts will be used to indicate if $C$ depends on any other parameters.

\begin{lemma}\label{lem:subg_moments}
If $P$ is $\sigma^2$ subgaussian, then
\begin{equation*}
E_P \|X\|^{2k} \leq (2 d \sigma^2)^k k!
\end{equation*}
for all nonnegative integers $k$, and
\begin{equation*}
E_P e^{v \cdot X} \leq E_P e^{\|v\| \|X\|} \leq 2 e^{\frac{d \sigma^2}{2} \|v\|^2}
\end{equation*}
for all $v \in \RR^d$.
\end{lemma}
\begin{proof}
For the first claim, it suffices to take expectations of both sides of the inequality \mbox{$\frac{\|X\|^{2k}}{(2 d \sigma^2)^k k!} \leq e^{\frac{\|X\|^2}{2 d \sigma^2}} - 1$} and use the assumption that $P$ is $\sigma^2$-subgaussian.
To prove the second claim, we use the inequality $v \cdot X \leq \|v\| \|X\| \leq \frac{d \sigma^2}{2}\|v\|^2 + \frac{1}{2 d \sigma^2} \|X\|^2$ and apply subgaussianity.
\end{proof}

 \begin{lemma}\label{lem:uniform_subg}
 Suppose $P$ is a $\sigma^2$-subgaussian measure. Then, there exists a (random) $\sigma_u < \infty$ such that $\{P_n\},P$ are uniformly $\sigma_u^2$-subgaussian $P$ almost surely.\label{usubgaussian}
 \end{lemma}
 \begin{proof} By definition, there exists $\sigma>0$ such that \mbox{$E_P\left(e^{\frac{||X||^2}{2\sigma^2d}}\right)\leq 2$}. By the strong law of large numbers we have that $P$ almost surely
 $$\lim_{n\rightarrow} E_{P_n} \left(e^{\frac{||X||^2}{2d\sigma^2}}\right)=E_P\left(e^{\frac{||X||^2}{2\sigma^2d}}\right)\leq 2\,.$$
 In particular, this implies the sequence $E_{P_n} \left(e^{\frac{||X||^2}{2\sigma^2d}}\right)$ is bounded by a random positive number. By the equivalence of definitions of subgaussianity, this implies that $P_n$ are uniformly subgaussian, with a new parameter that we call $\sigma_u^2$.
 \end{proof}

\begin{lemma}\label{lem:moment_bound}
Let $\mu_\beta$ be defined as in \eqref{eq:generalized_moments}.
Then
\begin{equation*}
|\mu_{\beta}| \leq C_{|\beta|, d} \left\{
\begin{array}{ll}
\sigma^{|\beta|}(\sigma+\sigma^2)^{|\beta|} & \|x\| \leq \sqrt d \sigma \\
\sigma^{|\beta|}(\sqrt{\sigma\|x\|} + \sigma \|x\|)^{|\beta|} &  \|x\| > \sqrt d \sigma\,.
\end{array}\right.
\end{equation*}
\end{lemma}
\begin{proof}
To bound $\mu_\beta$, we split the integral in the numerator according to the norm of $y$.
Let \mbox{$A = \{y : \|y\| \leq \tau\}$}, where $\tau$ is a threshold to be chosen.
Then
\begin{equation*}
\mu_\beta = \frac{ \int \mathds{1}_{A} y^\beta e^{g(y) - \frac 12 \|y\|^2 + x \cdot y} \, \dd Q(y)}{\int e^{g(y) - \frac 12 \|y\|^2 + x \cdot y} \, \dd Q(y)} + \frac{ \int \mathds{1}_{\overline A} y^\beta e^{g(y) - \frac 12 \|y\|^2 + x \cdot y} \, \dd Q(y)}{\int e^{g(y) - \frac 12 \|y\|^2 + x \cdot y} \, \dd Q(y)}\,.
\end{equation*}
The first term is clearly bounded by $\tau^\beta$.
For the second, we apply Proposition~\ref{extend-potentials} to show
\begin{equation*}
\left(\int e^{g(y) - \frac 12 \|y\|^2 + x \cdot y} \, \dd Q(y)\right)^{-1} = e^{-\frac 12 \|x\|^2}e^{f(x)} \leq e^{d \sigma^2 + \sqrt d \sigma \|x\|}
\end{equation*}
and
\begin{equation*}
e^{g(y) - \frac 12 \|y\|^2} \leq e^{d \sigma^2 + \sqrt d \sigma \|y\|}\,.
\end{equation*}
We obtain
\begin{align*}
\frac{ \int \mathds{1}_{\overline A} y^\beta e^{g(y) - \frac 12 \|y\|^2 + x \cdot y} \, \dd Q(y)}{\int e^{g(y) - \frac 12 \|y\|^2 + x \cdot y} \, \dd Q(y)} & \leq e^{2d \sigma^2 + \sqrt d \sigma \|x\|} \int \1_{\overline A} y^\beta e^{\sqrt d \sigma \|y\| + x \cdot y} \, \dd Q(y) \\
& \leq e^{2d \sigma^2 + \sqrt d \sigma \|x\|} \left(\int \1_{\overline A} y^{2 \beta} \, \dd Q(y) \right)^{1/2} \left(\int e^{2(\sqrt d \sigma +  \|x\|) \|y\|} \, \dd Q(y)\right)^{1/2}
\end{align*}
Since $Q$ is subgaussian, Lemma~\ref{lem:subg_moments} in Appendix \ref{app:technical} and the definition of $A$ imply
\begin{equation*}
\left(\int \1_{\overline A} y^{2 \beta} \, \dd Q(y) \right)^{1/2} \leq 
e^{- \frac{\tau^2}{8 d \sigma^2}} \left(\int e^{\frac{\|y\|^2}{4 d \sigma^2}} y^{2 \beta} \, \dd Q(y) \right)^{1/2} \leq \sqrt 2 e^{- \frac{\tau^2}{8 d \sigma^2}} (2|\beta|)!^{1/4} (\sqrt{2 d}\sigma)^{|\beta|}\,.
\end{equation*}
Lemma~\ref{lem:subg_moments} in Appendix \ref{app:technical} also implies
\begin{equation*}
\int e^{2(\sqrt d \sigma +  \|x\|) \|y\|} \, \dd Q(y) \leq 2 e^{2 d \sigma^2 (\|x\| + \sqrt d \sigma)^2}\,.
\end{equation*}

Therefore, if we choose $\tau^2 \geq C_{|\beta|,d} (\sigma^4 + \sigma^6)$ if $\|x\| \leq \sqrt d \sigma$ and $\tau^2 \geq C_{|\beta|,d} (\sigma^3\|x\| + \sigma^4 \|x\|^2)$ if $\|x\| > \sqrt d \sigma$ for a sufficiently large constant $C_{|\beta|,d}$, then we will have
\begin{equation*}
\frac{ \int \mathds{1}_{\overline A} y^\beta e^{g(y) - \frac 12 \|y\|^2 + x \cdot y} \, \dd Q(y)}{\int e^{g(y) - \frac 12 \|y\|^2 + x \cdot y} \, \dd Q(y)} \leq C_{|\beta|, d} (\sqrt d \sigma)^{|\beta|}
\end{equation*}
Combining this with the bound on the first term yields the claim.
\end{proof}

\begin{lemma}\label{lem:proxy_moments}
Let $\tilde \sigma$ be defined as in the proof of Theorem~\ref{theo1}.
Then for any positive integer $k$,
\begin{equation*}
E \tilde \sigma^{2k} \leq 2 k^k \sigma^{2k}\,.
\end{equation*}
\end{lemma}
\begin{proof}
First, let $P$ be an arbitrary probability distribution, and let $\alpha > 0$.
We first show that if $t = E_P e^{\frac{\|X\|^2}{\alpha}}$ is finite, then $P$ is $t \frac{\alpha}{2d}$-subgaussian.
To see this, set $\tau^2 = t \frac{\alpha}{2d}$.
Then
\begin{equation*}
E e^{\frac{\|X\|^2}{2 d \tau^2}} \leq \left(E e^{\frac{\|X\|^2}{\alpha}}\right)^\frac{\alpha}{2 d \tau^2} = t^{1/t} \leq e^{1/e} < 2\,,
\end{equation*}
where the first step uses Jensen's inequality and the fact that $t \geq 1$.

The above considerations imply that if $Q$ is $\sigma^2$ subgaussian and we set
\begin{equation*}
\tau^2 = \max\{E_{P_n} e^{\frac{\|X\|^2}{2kd \sigma^2}} k \sigma^2, E_{Q_n} e^{\frac{\|Y\|^2}{2kd \sigma^2}} k \sigma^2\}\,,
\end{equation*}
then $P_n$, $Q_n$, $P$, and $Q$ are all $\tau^2$ subgaussian, which implies that $\tilde \sigma^2 \leq \tau^2$.

Therefore, by Jensen's inequality,
\begin{equation*}
\tilde \sigma^{2k} \leq E_{P_n} e^{\frac{\|X\|^2}{2d \sigma^2}} k^k  \sigma^{2k} +  E_{Q_n} e^{\frac{\|Y\|^2}{2d \sigma^2}} k^k  \sigma^{2k}\,,
\end{equation*}
and taking expectations with respect to $P$ and $Q$ yields
\begin{equation*}
E \tilde \sigma^{2k} \leq E_P e^{\frac{\|X\|^2}{2d \sigma^2}} k^k \sigma^{2k}+ E_Q  e^{\frac{\|Y\|^2}{2d \sigma^2}} k^k \sigma^{2k} \leq 4 k^k \sigma^{2k}\,.
\end{equation*}
\end{proof}

\bibliographystyle{apalike}
\bibliography{bibliography}

\end{document}